\documentclass[12pt]{article}
\usepackage{amsmath,amsthm,amsfonts,amssymb,epsfig}

\usepackage{graphicx}

\usepackage{subcaption}

\usepackage{tkz-graph}

\usepackage{tikz}

\newcommand{\deter}{\operatorname{det}}

\newtheorem{stel}{Theorem}

\newtheorem{lemma}{Lemma}
\theoremstyle{remark}

\begin{document}

\title{An elementary proof of a matrix tree theorem for directed graphs}

\author{Patrick De Leenheer\footnote{Department of Mathematics and Department of Integrative Biology, Oregon State University, Supported in part by NSF-DMS-1411853, deleenhp@math.oregonstate.edu}}

\date{}

\maketitle

\begin{abstract}
We present an elementary proof of a generalization of Kirchoff's matrix tree theorem to directed, weighted graphs. The proof is based on a specific factorization of the 
Laplacian matrices associated to the graphs, which only involves the two incidence matrices that capture the graph's topology. We also point out how this result can be used to calculate principal eigenvectors of the Laplacian matrices.  
\end{abstract}



\section{Introduction}
Kirchoff's matrix tree theorem \cite{kirchoff} is a result that allows one to determine the number of spanning trees rooted at any vertex of an undirected graph by simply computing 
the determinant of an appropriate matrix associated to the graph. A recent elementary proof of Kirchoff's matrix tree theorem can be found in \cite{stanley}. Kirchoff's 
matrix theorem can be extended to directed graphs, see \cite{chaiken,tutte-book} for proofs. According to \cite{chaiken}, an early proof of this extension is due to \cite{borchardt}, although the result is often attributed to Tutte in \cite{tutte1948}, and hence referred to as Tutte's Theorem. The goal of this paper is to provide an elementary proof of Tutte's Theorem. 
Most proofs of Tutte's Theorem appear to be based on the Leibniz formula for the determinant of a matrix expressed as a sum over 
permutations of the matrix elements. The strategy proposed here is to instead factor the matrix as a product of two rectangular matrices, and use Binet-Cauchy's determinant formula instead.

Let $G=(V,E)$ be a {\it directed graph} with finite vertex set $V=\{v_1,\dots, v_p\}$ and finite directed edge set $E=\{e_1,\dots, e_q\}$. We assume that each directed edge points from some vertex $v_i$ 
to another vertex $v_j\neq v_i$, and that there exists at most one directed edge from any vertex to any distinct vertex. 
For each vertex $v_i$ of $G$, we define the {\it in-degree of $v_i$} as the number of distinct directed edges which point to $v_i$. Similarly, the {\it out-degree of $v_i$} is the number of 
distinct directed edges of $G$ which point from $v_i$ to some other vertex. A {\it directed  cycle} of $G$ is a collection of distinct vertices $\{v_{i_1},v_{i_2},\dots, v_{i_n}\}$, and a collection of distinct directed edges 
$\{e_{k_1},\dots,e_{k_n}\}$ such that each $e_{k_j}$ points from $v_{i_j}$ to $v_{i_{j+1}}$, and where $v_{i_{n+1}}=v_{i_1}$.

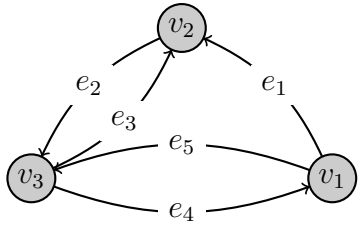
\begin{figure}
\centering
\begin{tikzpicture}[colorstyle/.style={circle, draw=black,fill=black!20,thick, inner sep=2pt, minimum size=1 mm, outer sep=0pt},
    scale=2]
    
\node (3) at (0,0)[colorstyle]{$v_3$};
\node (2) at (1,1)[colorstyle]{$v_2$};
\node (1) at (2,0)[colorstyle]{$v_1$};

\tikzset{EdgeStyle/.append style = {->,bend right = 20}}
 
\Edge[label = $e_1$](1)(2)
\Edge[label = $e_2$](2)(3)
\Edge[label = $e_3$](3)(2)
\Edge[label = $e_4$](3)(1)
\Edge[label = $e_5$](1)(3)

\end{tikzpicture}
\caption{An example of a directed graph $G$ with vertex set $V=\{v_1,v_2,v_3\}$ and directed edge set $E=\{e_1,e_2,e_3,e_4,e_5\}$.}\label{fig-running}
\end{figure}

{\bf Example}: Consider the directed graph in Figure $\ref{fig-running}$ which we shall use as a running example throughout this paper to illustrate the various concepts and notions. This directed graph has $p=3$ vertices and $q=5$ directed edges. The in-degrees of $v_1,v_2$ and $v_3$ are equal to $1,2$ and $2$ respectively. 
The out-degrees of $v_1,v_2$ and $v_3$ are equal to $2,1$ and $2$ respectively. 
There are several directed cycles, such as $\{v_2,v_3\}$ and $\{e_2,e_3\}$; $\{v_3,v_1\}$ and $\{e_4,e_5\}$; $\{v_1,v_2,v_3\}$ and $\{e_1,e_2,e_4\}$.

{\bf Definitions}: A {\it directed subgraph} of $G$ is a directed graph $G'=(V',E')$ with $V'\subseteq V$ and $E'\subseteq E$. 
Fix a vertex $v_r$ in $V$. We say that a directed subgraph $G'$ of $G$ is an  
{\it outgoing (incoming) directed spanning tree rooted at  $v_r$},  if $V'=V$, and if the following 3 conditions hold:
\begin{enumerate}
\item
Every vertex $v_i\neq v_r$ in $V'$ has in-degree (out-degree) 1.
\item
The root vertex $v_r$ has in-degree (out-degree) 0.
\item 
$G'$ has no directed cycles.
\end{enumerate}

Note that any outgoing (incoming) directed spanning tree of $G$ necessarily has $p-1$ distinct directed edges selected among the $q$ directed edges of $G$. Indeed, an outgoing (incoming) directed spanning tree must have exactly $p$ vertices. All its vertices except for the root $v_r$ must have in-degree (out-degree) equal to $1$, and the in-degree (out-degree) of the root $v_r$ must be $0$. Therefore, to identify outgoing (incoming) directed spanning trees, we should only consider directed subgraphs $G'$ of $G$ with the same number of vertices as $G$ (namely, $p$), and with exactly $p-1$ distinct directed edges chosen among the directed edges of $G$. There are a total of ${q\choose p-1}$  directed subgraphs $G'$ of $G$ with $p$ vertices and $p-1$ directed edges,  a possibly large number. Only some -and in some cases, none- of these 
directed subgraphs are outgoing (incoming) directed spanning trees, namely those which do not contain directed cycles. In order to count the number of outgoing (incoming) directed spanning trees at a given root, we need a counting scheme that recognizes the outgoing (incoming) directed spanning trees, and ignores directed subgraphs which contain directed cycles. The proof of Tutte's Theorem presented below provides such a counting scheme.

{\bf Example}: Consider the directed graph in Figure $\ref{fig-running}$, and choose as root $v_r=v_3$. There are $2$ outgoing (incoming) directed spanning trees rooted at $v_3$, which are depicted in Figure $\ref{trees-running-example1}$ (Figure $\ref{trees-running-example2}$). 

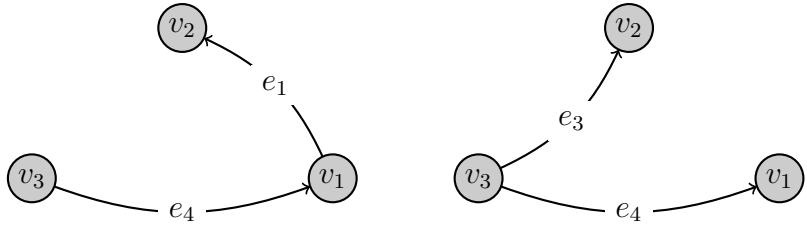
\begin{figure}
\centering
\begin{tikzpicture}[colorstyle/.style={circle, draw=black,fill=black!20,thick, inner sep=2pt, minimum size=1 mm, outer sep=0pt},scale=2]

\node (3) at (0,0)[colorstyle]{$v_3$};
\node (2) at (1,1)[colorstyle]{$v_2$};
\node (1) at (2,0)[colorstyle]{$v_1$};

\tikzset{EdgeStyle/.append style = {->,bend right = 20}}
 
\Edge[label = $e_1$](1)(2)
\Edge[label = $e_4$](3)(1)
\end{tikzpicture}
\hspace{1cm}
\begin{tikzpicture}[colorstyle/.style={circle, draw=black,fill=black!20,thick, inner sep=2pt, minimum size=1 mm, outer sep=0pt},scale=2]

\node (3) at (0,0)[colorstyle]{$v_3$};
\node (2) at (1,1)[colorstyle]{$v_2$};
\node (1) at (2,0)[colorstyle]{$v_1$};

\tikzset{EdgeStyle/.append style = {->,bend right = 20}}
 
\Edge[label = $e_3$](3)(2)
\Edge[label = $e_4$](3)(1)
\end{tikzpicture}
\caption{Outgoing directed spanning trees rooted at $v_r=v_3$ for the directed graph from Figure $\ref{fig-running}$.}
\label{trees-running-example1}
\end{figure}

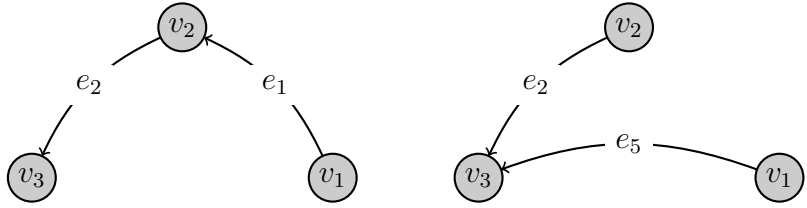
\begin{figure}
\centering
\begin{tikzpicture}[colorstyle/.style={circle, draw=black,fill=black!20,thick, inner sep=2pt, minimum size=1 mm, outer sep=0pt},scale=2]

\node (3) at (0,0)[colorstyle]{$v_3$};
\node (2) at (1,1)[colorstyle]{$v_2$};
\node (1) at (2,0)[colorstyle]{$v_1$};

\tikzset{EdgeStyle/.append style = {->,bend right = 20}}
 
\Edge[label = $e_1$](1)(2)
\Edge[label = $e_2$](2)(3)
\end{tikzpicture}
\hspace{1cm}
\begin{tikzpicture}[colorstyle/.style={circle, draw=black,fill=black!20,thick, inner sep=2pt, minimum size=1 mm, outer sep=0pt},scale=2]

\node (3) at (0,0)[colorstyle]{$v_3$};
\node (2) at (1,1)[colorstyle]{$v_2$};
\node (1) at (2,0)[colorstyle]{$v_1$};

\tikzset{EdgeStyle/.append style = {->,bend right = 20}}
 
\Edge[label = $e_2$](2)(3)
\Edge[label = $e_5$](1)(3)
\end{tikzpicture}
\caption{Incoming directed spanning trees rooted at $v_r=v_3$ for the directed graph from Figure $\ref{fig-running}$.}
\label{trees-running-example2}
\end{figure}

To a directed graph $G$ we associate two real $p\times p$ matrix, called the {\it Laplacians of $G$}, which are defined as follows:
\begin{equation}\label{laplacian}
L_1=D_{in}-A_v,\textrm{ and }L_2=D_{out}-A_v^T
\end{equation}
where:
\begin{itemize}
\item
 $D_{in}$ is a diagonal matrix defined as $[D_{in}]_{ii}=$ in-degree of vertex $v_i$, for all $i=1,\dots ,p$, 
\item
$A_v$ is the {\it vertex-adjacency matrix} of $G$, a real $p\times p$ matrix defined entry-wise as follows:
$$
[A_v]_{ij}=\begin{cases}1,\textrm{ if  there exists a  directed edge from } v_i \textrm{ to  }v_j\\
0,\textrm{ otherwise}
\end{cases}
$$
\item
 $D_{out}$ is a diagonal matrix defined as $[D_{out}]_{ii}=$ out-degree of vertex $v_i$, for all $i=1,\dots ,p$, 
\end{itemize}

Fix a vertex $v_r$ in $G$, and define the {\it reduced Laplacians $L_1^r$ and $L_2^r$}, by removing the $r$th row and $r$th columns from $L_1$ and $L_2$ respectively. Then 
Tutte's Theorem asserts that:
\begin{stel} \label{tutte} ({\bf Tutte's Theorem}) Let $G=(V,E)$ be a directed graph. The
number of outgoing and incoming directed spanning trees rooted at  $v_r$ are equal to $\deter(L_1^r)$ and $\deter(L_2^r)$ respectively.
\end{stel}

{\bf Example}:  For the directed graph from Figure $\ref{fig-running}$, and picking the root $v_r=v_3$, we have that:
$$
D_{in}=\begin{pmatrix}1&0&0\\ 0&2&0 \\0&0&2 \end{pmatrix},\textrm{ and } A_v=\begin{pmatrix}0&1&1\\0&0&1\\ 1&1&0 \end{pmatrix},
$$
and thus
$$
L_1=\begin{pmatrix}
1&-1&-1\\
0&2&-1\\
-1&-1&2
\end{pmatrix},\textrm{ and } L_1^r=\begin{pmatrix}1&-1\\0&2 \end{pmatrix}.
$$
Then $\deter(L_1^r)=2$ is indeed equal to the number of outgoing directed spanning trees rooted at $v_r=v_3$, confirming Tutte's Theorem for this example. 
Similarly, 
$$
D_{out}=\begin{pmatrix}2&0&0\\ 0&1&0 \\0&0&2 \end{pmatrix},
$$
and thus
$$
L_2=\begin{pmatrix}2&0&-1\\ -1&1&-1 \\-1&-1&2 \end{pmatrix},\textrm{ and } L_2^r=\begin{pmatrix}2&0\\-1&1 \end{pmatrix}.
$$
Then $\deter(L_2^r)=2$ is also indeed equal to the number of incoming directed spanning trees rooted at $v_r=v_3$, once again confirming Tutte's Theorem for this example.

\section{Proof of Tutte's Theorem}
Although Tutte's Theorem is a remarkable result, it is expressed in terms of  rather complicated matrices associated to a directed graph, namely the reduced Laplacians.
To a directed graph, one can associate two much more elementary matrices, known as {\it incidence matrices},  which arise quite naturally. 
For a given vertex, one can record the directed edges pointing {\it to} this vertex. This information will be captured by one of the incidence matrices, namely by $N_{in}$. 
Similarly, one can record for each vertex, the directed edges pointing {\it from} this vertex, and this will be captured by the second incidence matrix $M_{out}$.



{\bf Definitions}: Let $G=(V,E)$ be a directed graph.
The {\it incidence matrix $N_{in}$} is a real $q\times p$ matrix defined entry-wise as follows:
$$
[N_{in}]_{ki}=\begin{cases}
1,\textrm{ if directed edge } e_k \textrm{ points {\it to} vertex } v_i, \hspace{1cm}
\begin{tikzpicture}[colorstyle/.style={circle, draw=black,fill=black!20,thick, inner sep=2pt, minimum size=1 mm, outer sep=0pt},
    scale=2]    
\node (1) at (0,0)[colorstyle]{};
\node (2) at (1,0)[colorstyle]{$v_i$};
\tikzset{EdgeStyle/.append style = {->,bend right = 0}}
\Edge[label = $e_k$](1)(2)
\end{tikzpicture}
\\
0,\textrm{ otherwise}
\end{cases}
$$
One can identify the $k$th row of $N_{in}$ with edge $e_k$. This row has exactly one non-zero entry which equals $1$ and is located in the $i$th column, where $v_i$ is the vertex 
to which edge $e_k$ points. 

Similarly, the i{\it ncidence matrix $M_{out}$} is a real $p\times q$ matrix defined entry-wise as:
$$
[M_{out}]_{ik}=\begin{cases}
1,\textrm{ if directed edge } e_k \textrm{ points {\it from} vertex } v_i, \hspace{1cm}
\begin{tikzpicture}[colorstyle/.style={circle, draw=black,fill=black!20,thick, inner sep=2pt, minimum size=1 mm, outer sep=0pt},
    scale=2]    
\node (1) at (0,0)[colorstyle]{$v_i$};
\node (2) at (1,0)[colorstyle]{};
\tikzset{EdgeStyle/.append style = {->,bend right = 0}}
\Edge[label = $e_k$](1)(2)
\end{tikzpicture}
\\
0,\textrm{ otherwise}
\end{cases}
$$
One can identify the $k$th column of $M_{out}$ with edge $e_k$. This column has exactly one non-zero entry which equals $1$ and is located in the $i$th row, where $v_i$ is the vertex  from which $e_k$ points. 

{\bf Example}:  For the directed graph from Figure $\ref{fig-running}$, 
$$
N_{in}=\begin{pmatrix}
0&1&0\\
0&0&1\\
0&1&0\\
1&0&0\\
0&0&1
\end{pmatrix},\textrm{ and }
M_{out}=\begin{pmatrix}
1&0&0&0&1\\
0&1&0&0&0\\
0&0&1&1&0
\end{pmatrix}.
$$
The two incidence matrices contain purely {\it local} information concerning a directed graph, namely they record which directed edges point to, respectively from each vertex. 
But on the other hand, both matrices also provide us with {\it global} information about the graph. Indeed, given these two matrices, we can unambiguously construct the graph. 
This suggests that perhaps the two Laplacians of a directed graph can be expressed in terms of just  the two incidence matrices. The following factorization result shows that this is indeed the case.
\begin{lemma}\label{factors} Let $G=(V,E)$ be a directed graph. Then 
\begin{equation}\label{DenA}
D_{in}=N_{in}^TN_{in},\; A_v=M_{out}N_{in}  \textrm{ and } D_{out}=M_{out}M^T_{out},
\end{equation}
and thus
\begin{equation}\label{factorization}
L_1=(N_{in}^T-M_{out})N_{in},\textrm{ and }L_2=(M_{out}-N_{in}^T)M^T_{out}.
\end{equation}
Fix a vertex $v_r$ in $V$. Let $N_{in}^r$  be the matrix obtained from $N_{in}$ by removing the $r$th column in $N_{in}$, and let $M_{out}^r$ be the matrix obtained from $M_{out}$ be removing the $r$th row from $M_{out}$. Then
\begin{equation}\label{factorization-reduced}
L_1^r=((N^r_{in})^T-M_{out}^r)N_{in}^r,\textrm{ and }L_2^r=(M_{out}^r-(N^r_{in})^T)(M^r_{out})^T
\end{equation}
\end{lemma}
\begin{proof}
For all $i,j=1,2,\dots, p$, consider
\begin{eqnarray*}
[N^T_{in}N_{in}]_{ij}&=&\sum_{k=1}^q[N^T_{in}]_{ik}[N_{in}]_{kj}=\sum_{k=1}^q [N_{in}]_{ki}[N_{in}]_{kj}\\
&=&\begin{cases}0,\textrm{ if } i\neq j\\
\textrm{in-degree of vertex } v_i,\textrm{ if } i=j
\end{cases}
\end{eqnarray*}
Indeed, for each $k$, $[N_{in}]_{ki}[N_{in}]_{kj}=0$ when $i\neq j$ since each edge $e_k$ points to exactly one vertex, so at most one of the two factors in this product can be non-zero. 
On the other hand, when $i=j$, then $[N_{in}]_{ki}[N_{in}]_{kj}=\left([N_{in}]_{ki}\right)^2=1$ when $e_k$ points to $v_i$, but equals $0$ if $e_k$ does not point to $v_i$. The sum over 
all $q$ of these terms therefore yields the in-degree of vertex $v_i$. This establishes that $D_{in}=N^T_{in}N_{in}$. A similar proof shows that $D_{out}=M_{out}M^T_{out}$.

For all $i,j=1,2,\dots, p$, consider
\begin{eqnarray*}
[M_{out}N_{in}]_{ij}&=&\sum_{k=1}^q[M_{out}]_{ik}[N_{in}]_{kj}\\
&=&\begin{cases}0,\textrm{ if there is no edge pointing from } v_i \textrm{ to } v_j \\
1,\textrm{ if there is an edge pointing from } v_i\textrm{ to } v_j
\end{cases}
\end{eqnarray*}
Indeed, $[M_{out}]_{ik}[N_{in}]_{kj}=1$ precisely when edge $e_k$ points from vertex $v_i$ to vertex $v_j$, and equals zero otherwise. The sum over all 
$q$ of these terms can not be larger than $1$ because there is at most one distinct directed edge pointing from one vertex to another vertex. This establishes that $A_v=M_{out}N_{in}$. 

The definitions of $L_1$ and $L_2$, together with $(\ref{DenA})$, imply $(\ref{factorization})$. Finally, it is immediate from the definition of $L_1$ ($L_2$) that $L_1^r=D_{in}^r-A_v^r$ 
($L_2^r=D_{out}^r-(A_v^T)^r$), where 
$D_{in}^r$ ($D_{out}^r$) and $A_v^r$ ($(A_v^T)^r$) are obtained from respectively $D_{in}$ ($D_{out}$) and $A_v$ ($A_v^T$) by deleting the $r$th row and $r$th column from these matrices. Moreover, $(\ref{DenA})$ implies that 
\begin{eqnarray*}
D_{in}^r=(N^r_{in})^TN^r_{in}&\textrm{ and }& A_v^r=M^r_{out}N^r_{in},\textrm{ and}\\
D_{out}^r=(M^r_{out})(M^r_{out})^T&\textrm{ and }& (A_v^T)^r=(N^r_{in})^T(M^r_{out})^T,
\end{eqnarray*}
which in turn yields $(\ref{factorization-reduced})$.
\end{proof}

We are now ready for a\\[2ex] 
{\bf Proof of Theorem $\ref{tutte}$}: We shall only provide a proof for the reduced Laplacian $L_1^r$ because the proof for the reduced Laplacian $L_2^r$ is analogous. Lemma $\ref{factors}$ implies that:
$$
\deter(L_1^r)=\deter \left( ((N^r_{in})^T-M_{out}^r)N_{in}^r\right).
$$

For notational convenience we set
$$
B=(N^r_{in})^T-M_{out}^r,\textrm{ and }C=N_{in}^r.
$$
Then Binet-Cauchy's determinant formula implies that:
$$
\deter(L_1^r)=\sum_{S\subseteq \{1,\dots, q\}, |S|=p-1} \deter(B[S])\deter(C[S]),
$$
where the sum is over all subsets $S$ of $\{1,\dots, q\}$ containing $p-1$ elements. There are ${q \choose p-1}$ such subsets. Furthermore, $B[S]$ denotes the $(p-1)\times(p-1)$ submatrix obtained from the 
$(p-1)\times q$ matrix $B$ by selecting precisely those columns of $B$ in the set $S$. Similarly, $C[S]$ is the submatrix obtained from the $q\times (p-1)$ matrix $C$, by selecting precisely those 
rows of $C$ in the set $S$.

To complete the proof of Tutte's Theorem, we will show that:
\begin{enumerate}
\item When the $p-1$ elements in $S$ correspond to the indices of the directed edges of an outgoing directed spanning tree rooted at $v_r$, then
$$
\deter(B[S])\deter(C[S])=1.
$$
\item
When the $p-1$ elements in $S$ correspond to the indices of the directed edges of a directed subgraph of $G$ which is not an outgoing directed spanning tree rooted at $v_r$, then
$$
\deter(B[S])\deter(C[S])=0.
$$
\end{enumerate}

\noindent
1. Suppose that $S=\{k_1,k_2,\dots, k_{p-1}\}$, with $k_1<k_2<\dots, k_{p-1}$, is in a bijective correspondence to a set of $p-1$ indices of the directed edges $e_{k_1},e_{k_2},\dots , e_{k_{p-1}}$ of an outgoing directed spanning tree rooted at $v_r$.  Similarly, let ${\tilde S}=\{l_1,l_2,\dots, l_{p-1}\}$, with $l_1<l_2<\dots < l_{p-1}$ be in a bijective correspondence 
to the set of indices of the vertices in $V\setminus\{v_r\}$. 
Set $T=(V,E')$ to denote the directed subgraph of $G$ corresponding to this tree, i.e. 
$E'= \{ e_{k_1},e_{k_2},\dots , e_{k_{p-1}}\}$.

{\bf Claim}: The $(p-1)\times (p-1)$ matrix $C[S]=N^r_{in}[S]$ has precisely one non-zero entry in each row and in each column, and this non-zero entry equals $1$.
This follows from the fact that for an outgoing  
directed spanning tree rooted at $v_r$, each of the vertices distinct from $v_r$ has in-degree equal to $1$ (implying that each column of $C[S]=N^r_{in}[S]$ contains exactly one non-zero entry that equals $1$), and each of the $p-1$ directed edges points to one of the $p-1$ non-root vertices (implying that each row of $C[S]=N^r_{in}[S]$ contains exactly one non-zero entry that equals $1$). 

Consequently, the matrix $C[S]=N^r_{in}[S]$ is a permutation matrix, i.e. it is a matrix obtained from the identity matrix by finitely many column swaps, and thus $C[S](C[S])^T=I=(C[S])^TC[S]$, whence:
\begin{eqnarray}
\deter(B[S])\deter(C[S])&=&\deter(B[S]C[S])\nonumber \\
&=&\deter \left(I- M^r_{out}[S]N^r_{in}[S]\right) \nonumber \\
&=&\deter(I-D) \label{jordan},
\end{eqnarray}
where 
$$
D=M^r_{out}[S]N^r_{in}[S].
$$
Recall from Lemma $\ref{factors}$ that $A_v=M_{out}N_{in}$. By reducing $N_{in}$ to $N_{in}^r$, and $M_{out}$ to $M_{out}^r$ as in Lemma $\ref{factors}$, and then selecting the respective sub-matrices $N_{in}^r[S]$ and $M^r_{out}[S]$, we obtain that:
$$
D_{ij}=\begin{cases}
1,\textrm{ if there is a directed edge in }E' \textrm{ pointing from } v_{l_i} \textrm{ to } v_{l_j}, \textrm{ both in } V\setminus \{v_r\}\\
0,\textrm{ otherwise }
\end{cases}
$$

{\bf Claim}: $D$ is nilpotent, and hence there is an invertible $(p-1)\times (p-1)$ matrix $S$ such that 
$S^{-1}DS=J$, where $J$, the Jordan canonical form of $D$, is strictly upper-triangular (all diagonal entries of $J$ are zero), and then $(\ref{jordan})$ 
implies that:
$$
\deter(B[S])\deter(C[S])=\deter \left(I-D \right)=\deter \left(S^{-1}(I-D)S \right)=\deter \left(I-J \right)=+1.
$$
To show that $D$ is nilpotent, we will prove that $D^{p-1}=0$. Arguing by contradiction, assume that $[D^{p-1}]_{i_1 i_p}\neq 0$ for some $i_1$ and $i_p$ in $\{1,\dots, p-1\}$. 
Then there exist $v_{l_{i_1}},v_{l_{i_2}},\dots, v_{l_{i_p}}$ in $V\setminus \{v_r\}$, such that for all $s=1,\dots, p-1$,  there is some directed edge in $E'$ from 
$v_{l_{i_s}}$ to $v_{l_{i_{s+1}}}$. Since $V\setminus \{v_r\}$ contains $p-1$ distinct elements, it follows from the Pigeonhole Principle that  
the sequence $v_{l_{i_1}},v_{l_{i_2}},\dots, v_{l_{i_p}}$ 
must contain two identical terms. But then $T$ contains a directed cycle, which is a contradiction.\\

\noindent
2. Suppose that $S=\{k_1,\dots, k_{p-1}\}$ with $k_1<\dots < k_{p-1}$ corresponds to the index set of a subgraph $G'=(V,E')$ of $G$ with $
E'=\{e_{k_1},\dots, e_{k_{p-1}}\}$, which  is not an outgoing directed spanning tree rooted at $v_r$. Our goal is to show that:
$$
\deter(B[S])\deter(C[S])=0.
$$
There are 3 possible cases to consider:\\
\noindent
{\bf Case 1}: There is a vertex $v_i$ in $G'$ with $v_i\neq v_r$, whose in-degree  is not $1$.\\
Then either the in-degree of $v_i$ is $0$, or it is at least $2$. 
If the in-degree of $v_i$ is $0$, then the column of $N_{in}^r[S]$ that records all in-coming edges to $v_i$ in $G'$, is a zero column vector, and thus $\deter(C[S])=\deter(N^r_{in}[S])=0$.
If the in-degree of $v_i$ is at least $2$, then there are at least two identical rows in the matrix $N_{in}^r[S]$, and likewise $\deter(C[S])=\deter(N^r_{in}[S])=0$.\\[2ex]
\noindent
{\bf Case 2}: The in-degree of the root $v_r$ is not $0$.\\
Then $N^r_{in}[S]$ has at least one zero row, and thus $\deter(C[S])=\deter(N^r_{in}[S])=0$.\\[2ex]
\noindent
{\bf Case 3}: $G'$ contains a directed cycle.\\
Let $\{v_{i_1},v_{i_2},\dots, v_{i_n}\}$ be a collection of distinct vertices in $V'=V$, and let 
$\{e_{l_1},\dots,e_{l_n}\}$  be a collection of distinct directed edges 
in $E'$ such that each $e_{l_j}$ points from $v_{i_j}$ to $v_{i_{j+1}}$, and where $v_{i_{n+1}}=v_{i_1}$. 
We claim that the sum of the $l_1$th, $l_2$th, $\dots$, $l_n$th columns of $B[S]=(N^r_{in})^T[S]-M^r_{out}[S]$ equals the zero vector, and thus 
$\deter(B[S])=\deter \left(N^r_{in})^T[S]-M^r_{out}[S]\right)=0$. To see why, note that each of the mentioned columns has exactly two non-zero entries, one being $+1$, and the other being $-1$. Moreover, the $l_j$th and the $l_{j+1}$th columns will have two non-zero entries of opposite sign in the same position. Therefore, by adding the columns, each $+1$ in some column is canceled by a $-1$ in another column, establishing the claim.

\section{Spanning trees and eigenvectors}
To a directed graph $G=(V,E)$, we have associated two Laplacians $L_1$ and $L_2$, see $(\ref{laplacian})$. The sum over all rows in both Laplacians is the zero vector; equivalently, all  column sums in both Laplacians are equal to zero. This follows from $(\ref{factorization})$ because
$$
(1,\, 1,\, \dots, 1) \left(N_{in}^T - M_{out} \right)=(1,\, 1,\, \dots, 1) \left(M_{out} - N^T_{in} \right)=0,
$$
since each column of the matrix $N_{in}^T - M_{out}$ contains exactly two non-zero entries, one being a $+1$ and the other a $-1$. Thus, zero is an eigenvalue of 
both $L_1$ and $L_2$. We will find eigenvectors associated to the zero eigenvalues of $L_1$ and $L_2$, and relate them to the number of outgoing, respectively incoming directed spanning trees rooted at each of the vertices of $G$.

\begin{stel}\label{eigen}
Let $G=(V,E)$ be a directed graph. Suppose that $x$ and $y$ are $p$-vectors, defined entry-wise as follows:
\begin{eqnarray}
x_i&=&\textrm{ number of outgoing directed spanning trees rooted at vertex } v_i, \textrm{ and }\\
y_i&=&\textrm{ number of incoming directed spanning trees rooted at vertex } v_i,
\end{eqnarray}
for all $i=1,\dots, p$. Then
\begin{equation}\label{e-vectors}
L_1x=0=L_2y.
\end{equation}
\end{stel}
\begin{proof}
We only give a proof for $L_2$ as it is similar for $L_1$. For notational convenience, we drop the subscript of $L_2$, and denote this matrix by $L$.
Since zero is an eigenvalue $L$, we have that $\deter(L)=0$. By expanding the determinant of $L$ along each of the rows of $L$, we see that:
\begin{eqnarray*}
L_{11}C_{11}+L_{12}C_{12}+\dots + L_{1p}C_{1p}&=&0\\
L_{21}C_{21}+L_{22}C_{22}+\dots +L_{2p}C_{2p}&=&0\\
&\vdots& \\
L_{p1}C_{p1}+L_{p2}C_{p2}+\dots +L_{pp}C_{pp}&=& 0,
\end{eqnarray*}
where $C_{ij}$ denotes the cofactor of $L_{ij}$.

We claim that for all $i,j,k$ in $\{1,\dots, p\}$:
$$
C_{ij}=C_{kj}.
$$
That is, the cofactors of elements of $L$ in the same column, are all equal. This is a standard exercise in linear algebra that relies on basic properties of determinants, and exploits the fact that all the column sums of $L$ equal zero, as remarked earlier. From Tutte's Theorem (Theorem $\ref{tutte}$)  follows that for all $i=1,\dots ,p$:
 $$
 C_{ii}=\deter(L^i)=y_i=\textrm{number of incoming directed spanning trees rooted at }v_i. 
$$
This implies that $Ly=0$. 
\end{proof}
{\bf Remark} The vectors $x$ and $y$ in Theorem $\ref{eigen}$ are only eigenvectors of $L_1$ and $L_2$, when they are non-zero vectors. This requires that 
$G$ should have at least one vertex such that there is an outgoing (or imcoming) directed spanning tree rooted at that vertex. A sufficient condition for this to happen is that $G$ is a {\it strongly connected}  directed graph. This means that from every vertex of $G$ there must exist a directed path to any other vertex of $G$. When $G$ is strongly connected, there are a positive number of incoming and outgoing directed spanning trees rooted at every vertex of $G$. Hence the vectors $x$ and $y$ are entry-wise positive vectors. This result also follows directly from the celebrated Perron-Frobenius Theorem applied to the irreducible Laplacian matrices $-L_1$ and $-L_2$. Note that these matrices have non-negative off-diagonal entries, and we already know that both have an eigenvalue at zero. This eigenvalue is a {\it principal eigenvalue}, meaning that every other eigenvalue has negative real part. 
The Perron-Frobenius Theorem then implies that both matrices have unique (up to multiplication by non-zero scalars), entry-wise positive eigenvectors associated to their zero eigenvalue. Theorem $\ref{eigen}$ above provides  a way to compute these eigenvectors. Indeed, in principle they can be found by simply counting the number of outgoing and incoming directed spanning trees rooted at every vertex of $G$. In other words, we have established a purely graphical procedure to compute eigenvectors of the zero principal eigenvalues of the Laplacian matrices.

\section{Extensions to weighted directed graphs}
In this Section we generalize the preceding results to weighted directed graphs.

Let $G_w=(V,E,W)$ be a {\it weighted directed graph}, where $V=\{v_1,\dots, v_p\}$ is the vertex set, $E=\{e_1,\dots, e_q\}$ the directed edge set, 
and $W=\{w_1,\dots, w_q\}$ is the set of positive weights associated to each of the directed edges. The {\it weight of a weighted directed graph} is defined as the product of the 
weights of its edges:
$$
\textrm{ weight of } G_w=\Pi_{i=1}^q w_i
$$

A weighted directed subgraph of $G_w$ is a weighted directed graph $G_w'=(V',E',W')$, where $V'\subseteq V$, $E'\subseteq E$, and $W'$ is the subset of $W$ that corresponds to the subset $E'$ of $E$. Fix a vertex $v_r$ in $V$. We say that a weighted directed  subgraph $G'_w$ is a {\it weighted outgoing (incoming) directed spanning tree rooted at $v_r$} if 
$G'=(V',E')$ is an outgoing (incoming) directed spanning tree rooted at $v_r$.

To a weighted directed graph $G_w$ we can associate two real $p\times p$ matrices, also called the {\it Laplacians of $G_w$}, which are defined as follows:
\begin{equation}\label{laplacian2}
L_{1,w}=D_{in,w}-A_{v,w},\textrm{ and }L_{2,w}=D_{out,w}-A_{v,w}^T,
\end{equation}
where:
\begin{itemize}
\item
 $D_{in,w}$ is a diagonal matrix such that for all $i=1,\dots ,p$, 
 $[D_{in,w}]_{ii}$ is equal to the sum of the weights of all incoming edges to vertex $v_i$.
  \item
$A_{v,w}$ is the {\it weighted vertex-adjacency matrix} of $G_w$, a real $p\times p$ matrix defined entry-wise as follows:
$$
[A_{v,w}]_{ij}=\begin{cases}w_k ,\textrm{ if  } e_k\textrm{ is the weighted directed edge from } v_i \textrm{ to  }v_j\\
0,\textrm{ if there is no weighted directed edge from } v_i \textrm{ to } v_j 
\end{cases}
$$
\item
 $D_{out,w}$ is a diagonal matrix such that for all $i=1,\dots ,p$, 
 $[D_{out,w}]_{ii}$ is equal to the sum of the weights of all outgoing edges of vertex $v_i$.
\end{itemize}
The {\it weighted incidence matrices} $N_{in,w}$ and $M_{out,w}$ can be associated to a weighted directed graph $G_w$ as well. The matrix $N_{in,w}$ is a $q\times p$ matrix defined 
entry-wise as follows:
$$
[N_{in,w}]_{ki}=\begin{cases}
w_k^{1/2},\textrm{ if directed edge } e_k \textrm{ points {\it to} vertex } v_i, \hspace{1cm}
\begin{tikzpicture}[colorstyle/.style={circle, draw=black,fill=black!20,thick, inner sep=2pt, minimum size=1 mm, outer sep=0pt},
    scale=2]    
\node (1) at (0,0)[colorstyle]{};
\node (2) at (1,0)[colorstyle]{$v_i$};
\tikzset{EdgeStyle/.append style = {->,bend right = 0}}
\Edge[label = $e_k$](1)(2)
\end{tikzpicture}
\\
0,\textrm{ otherwise}
\end{cases}
$$
and similarly, the matrix $M_{out,w}$ is a $p\times q$ matrix defined entry-wise as follows:
$$
[M_{out,w}]_{ik}=\begin{cases}
w_k^{1/2},\textrm{ if directed edge } e_k \textrm{ points {\it from} vertex } v_i, \hspace{1cm}
\begin{tikzpicture}[colorstyle/.style={circle, draw=black,fill=black!20,thick, inner sep=2pt, minimum size=1 mm, outer sep=0pt},
    scale=2]    
\node (1) at (0,0)[colorstyle]{$v_i$};
\node (2) at (1,0)[colorstyle]{};
\tikzset{EdgeStyle/.append style = {->,bend right = 0}}
\Edge[label = $e_k$](1)(2)
\end{tikzpicture}
\\
0,\textrm{ otherwise}
\end{cases}
$$
The key observation is that $L_{1,w}$ and $L_{2,w}$ can still be factored using the weighted incidence matrices, as in Lemma $\ref{factorization}$:
\begin{equation}\label{factorization2}
L_{1,w}=(N_{in,w}^T-M_{out,w})N_{in,w},\textrm{ and }L_{2,w}=(M_{out,w}-N_{in,w}^T)M^T_{out,w},
\end{equation}
Fix a vertex $v_r$ in the weighted directed graph $G_w$.
Denoting the reduced matrix $N_{in,w}^r$ ($M^r_{out,w}$) as the matrix obtained  by deleting the 
$r$th column ($r$th row) from $N_{in,w}$ ($M^r_{out,w}$), and $L^r_{1,w}$ ($L^r_{2,w}$) by deleting the $r$th row and the $r$th column 
from $L_{1,w}$ ($L_{2,w}$), we also have that:
\begin{equation}\label{factorization-reduced2}
L_1^r=((N^r_{in})^T-M_{out}^r)N_{in}^r,\textrm{ and }L_2^r=(M_{out}^r-(N^r_{in})^T)(M^r_{out})^T
\end{equation}
These factorizations enable a proof of a generalization of Tutte's Theorem (Theorem $\ref{tutte}$) to weighted graphs, as well as a 
generalization of Theorem $\ref{eigen}$:
\begin{stel} \label{tutte2} Let $G_w=(V,E,W)$ be a weighted directed graph. Then 
the sum of the weights of the weighted outgoing (incoming) directed spanning trees rooted at  $v_r$ is equal to $\deter(L_{1,w}^r)$ ($\deter(L_{2,w}^r)$).

Suppose that $x$ and $y$ are $p$-vectors, defined entry-wise as follows:
\begin{eqnarray*}
x_i&=&\textrm{ sum of weights of all weighted outgoing directed spanning trees rooted at } v_i, \textrm{ and }\\
y_i&=&\textrm{ sum of weights of all weighted incoming directed spanning trees rooted at } v_i,
\end{eqnarray*}
for all $i=1,\dots, p$. Then
\begin{equation}\label{e-vectors2}
L_{1,w}x=0=L_{2,w}y.
\end{equation}
\end{stel}
We shall skip the proof of Theorem $\ref{tutte2}$, and only highlight two instances where the proof of Theorem $\ref{tutte}$ needs to be modified slightly. 
The proof of Theorem $\ref{eigen}$ does not require any modifications. 
\begin{itemize}
\item
As in the first part of the proof of Theorem $\ref{tutte}$, we set
$$
B[S]=(N^r_{in,w})^T - M_{out}^r,\textrm{ and } C[S]=N^r_{in,w}.
$$
Here, $C[S]=N^r_{in,w}[S]$ is not necessarily a permutation matrix anymore, and $(\ref{jordan})$ is modified to:
$$
\deter(B[S])\deter(C[S])=\deter(Q-D),
$$
where 
$$
Q=(N^r_{in,w}[S])^T N^r_{in,w}[S],\textrm{ and }D=M^r_{out,w}[S]N^r_{in,w}[S].
$$
It can be shown that $Q$ is a diagonal matrix whose determinant equals the weight of the weighted spanning tree $T$. As before, $D$ is still 
a nilpotent matrix (because $D^{p-1}=0$), and thus 
$$
\deter(B[S])\deter(C[S])=\deter(Q-D)=\deter(Q)=\textrm{weight of } T.
$$
\item
Another modification must be made to the last paragraph of the proof of Theorem $\ref{tutte}$. We claim that the sum of the 
$l_1$th, $l_2$th, ... $l_n$th columns of $B[S]$ is still zero. 
In this case, each of the mentioned columns has exactly two non-zero entries, one being the square root of the weight of some directed edge, and the other 
being minus the square root of the weight of a directed edge (instead of a $+1$ and a $-1$). But as before, the $l_j$th and the $l_{j+1}$th columns have two non-zero entries of opposite sign in the same position, and consequently these columns add up to the zero vector. Thus, we still conclude that $\deter(B[S])=0$.
\end{itemize}

\end{document}